\documentclass[12pt]{amsart}

\usepackage{amssymb}

\usepackage{enumitem}

\usepackage{graphicx}
\usepackage[all,cmtip]{xy}

\makeatletter
\@namedef{subjclassname@2010}{%
  \textup{2010} Mathematics Subject Classification}
\makeatother

\usepackage[T1]{fontenc}

\newcommand{\bea}{\begin{eqnarray}}
\newcommand{\eea}{\end{eqnarray}}

\newcommand{\clb}{\mathcal{B}}

\newcommand{\cld}{\mathcal{D}}
\newcommand{\cle}{\mathcal{E}}

\newcommand{\clh}{\mathcal{H}}
\newcommand{\clk}{\mathcal{K}}

\newcommand{\cls}{\mathcal{S}}

\newcommand{\clw}{\mathcal{W}}

\newcommand{\D}{\mathbb{D}}

\newcommand{\raro}{\rightarrow}

\newcommand{\TP}{\tilde{\Pi}}

\newcommand{\mo}{\mathop{\oplus}}

\def\textmatrix#1&#2\\#3&#4\\{\bigl({#1 \atop #3}\ {#2 \atop #4}\bigr)}
\def\dispmatrix#1&#2\\#3&#4\\{\left({#1 \atop #3}\ {#2 \atop #4}\right)}
\newcommand{\be}{\begin{equation}}
\newcommand{\ee}{\end{equation}}
\newcommand{\ben}{\begin{eqnarray*}}
\newcommand{\een}{\end{eqnarray*}}

\newcommand{\NI}{\noindent}

\newcommand{\bi}{\begin{itemize}}
\newcommand{\ei}{\end{itemize}}

\newtheorem{theorem}{Theorem}[section]
\newtheorem{corollary}[theorem]{Corollary}
\newtheorem{lemma}[theorem]{Lemma}

\theoremstyle{definition}

\numberwithin{equation}{section}

\begin{document}


\baselineskip=17pt


\title[Pairs of Commuting Isometries - I]{Pairs of Commuting Isometries - I}

\author[A. Maji]{Amit Maji}
\address{Indian Statistical Institute\\ Statistics and Mathematics Unit\\ 8th Mile, Mysore Road\\ Bangalore, 560059, India}
\email{amit.iitm07@gmail.com}

\author[J. Sarkar]{Jaydeb Sarkar}
\address{Indian Statistical Institute\\ Statistics and Mathematics Unit\\ 8th Mile, Mysore Road\\ Bangalore, 560059, India}
\email{jay@isibang.ac.in, jaydeb@gmail.com}

\author[Sankar T. R.]{Sankar T. R.}
\address{Indian Statistical Institute\\ Statistics and Mathematics Unit\\ 8th Mile, Mysore Road\\ Bangalore, 560059, India}
\email{sankartr90@gmail.com}

\date{}

\begin{abstract}
We present an explicit version of Berger, Coburn and Lebow's
classification result for pure pairs of commuting isometries in the
sense of an explicit recipe for constructing pairs of commuting
isometric multipliers with precise coefficients. We describe a
complete set of (joint) unitary invariants and compare the Berger,
Coburn and Lebow's representations with other natural analytic
representations of pure pairs of commuting isometries. Finally, we
study the defect operators of pairs of commuting isometries.
\end{abstract}



\maketitle

\section{Introduction}
A very general and fundamental problem in the theory of bounded linear operators
on Hilbert spaces is to find classifications and representations of commuting
families of isometries.

In the case of single isometries this question has a complete and
explicit answer: If $V$ is an isometry on a Hilbert space $\clh$,
then there exist a Hilbert space $\clh_u$ and a unitary operator
$U$ on $\clh_u$ such that $V$ on $\clh$ and $\begin{bmatrix}S
\otimes I_{\clw} & 0
\\0 & U
\end{bmatrix}$ on $(l^2(\mathbb{Z}_+) \otimes {\clw}) \oplus \clh_u$
are unitarily equivalent, where $\clw = \ker V^*$ is the wandering
subspace for $V$ and $S$ is the forward shift operator on
$l^2(\mathbb{Z}_+)$ \cite{H}. This fundamental result is due to J.
von Neumann \cite{VN} and H. Wold \cite{W} (see Theorem
\ref{thm-Wold} for more details).

The case of pairs (and $n$-tuples) of commuting isometries is more
subtle, and is directly related to the commutant lifting theorem
\cite{FF} (in terms of an explicit, and then unique solution),
invariant subspace problem \cite{HH} and representations of
contractions on Hilbert spaces in function Hilbert spaces \cite{NF}.
For instance:

\NI \textsf{(a)} Let $\cls$ be a closed joint $(M_{z_1},
M_{z_2})$-invariant subspace of the Hardy space $H^2(\D^2)$. Then
$(M_{z_1}|_{\cls}, M_{z_2}|_{\cls})$ on $\cls$ is a pure (see
Section 3) pair of commuting isometries. Classification of such
pairs of isometries is largely unknown (see Rudin \cite{RUD}).

\NI \textsf{(b)} Let $T$ be a contraction on a Hilbert space $\clh$.
Then there exists a pair of commuting isometries $(V_1, V_2)$ on a
Hilbert space $\clk$ such that $T$ and $P_{\ker V_2^*} V_1|_{\ker
V_2^*}$ are unitarily equivalent (see Bercovici, Douglas and Foias
\cite{BDF3}).

\NI \textsf{(c)} The celebrated Ando dilation theorem (see Ando
\cite{AN}) states that a commuting pair of contractions dilates to a
commuting pair of isometries. Again, the structure of Ando's pairs
of commuting isometries is largely unknown.

The main purpose of this paper is to explore and relate various
natural representations of a large class of pairs of commuting
isometries on Hilbert spaces. The geometry of Hilbert spaces, the
classical Wold-von Neumann decomposition for isometries, the
analytic structure of the commutator of the unilateral shift, and
the Berger, Coburn and Lebow \cite{BCL} representations of pure
pairs of commuting isometries are the main guiding principles for
our study. The Berger, Coburn and Lebow theorem states that: Let
$(V_1, V_2)$ be a pair of commuting isometries on a Hilbert space
$\clh$, and let $V = V_1 V_2$ be a shift (or, a pure isometry - see
Section 2). Then there exist a Hilbert space $\clw$, an orthogonal
projection $P$ and a unitary operator $U$ on $\clw$ such that
\[
{\Phi_1}(z)= U^* (P + z P^{\perp}) \quad \text{and} \quad
{\Phi_2}(z) = (P^{\perp} + zP) U \quad \quad (z \in \D),
\]
are commuting isometric multipliers in $H^\infty_{\clb(\clw)}(\D)$,
and $(V_1, V_2, V)$ on $\clh$ and $(M_{\Phi_1}, M_{\Phi_2}, M_z)$ on
$H^2_{\clw}(\D)$ are unitarily equivalent (see Bercovici, Douglas
and Foias \cite{BDF3} for an elegant proof).

Here and further on, given a Hilbert space $\clh$ and a closed
subspace $\cls$ of $\clh$, $P_{\cls}$ denotes the orthogonal
projection of $\clh$ onto $\cls$. We also set
\[
P_{\cls}^{\perp} = I_{\clh} - P_{\cls}.
\]

In this paper we give a new and more concrete treatment, in the
sense of explicit representations and analytic descriptions, to the
structure of pure pairs of commuting isometries. More specifically,
we provide an explicit recipe for constructing the isometric
multipliers $(\Phi_1(z), \Phi_2(z))$, and the operators $U$ and $P$
involved in the coefficients of $\Phi_1$ and $\Phi_2$ (see Theorems
\ref{thm-BCLnew} and \ref{thm-BCLold}). Then we compare the Berger,
Coburn and Lebow representations with other possible analytic
representations of pairs of commuting isometries.

In Section 6, which is independent of the remaining part of the
paper, we analyze defect operators for (not necessarily pure) pairs
of commuting isometries. We provide a list of characterizations of
pairs of commuting isometries with positive defect operators (see
Theorem \ref{lemma21}). Our results hold in a more general setting
with somewhat simpler proofs (see Theorem \ref{thm-defect1} for
instance) than the one considered by He, Qin and Yang \cite{HQY}.
Moreover, we prove that for a large class of pure pairs of commuting
isometries the defect operator is negative if and only if the defect
operator is the zero operator.

The paper is organized as follows. In Section 2 we review the
classical Wold-von Neumann theorem for isometries and then prove a
representation theorem for commutators of shifts. In Section 3 we
discuss some basic relationships between wandering subspaces for
commuting isometries, followed by a new and explicit proof of the
Berger, Coburn and Lebow characterizations of pure pairs of
commuting isometries. Section 4 is devoted to a short discussion
about joint unitary invariants of pure pairs of commuting
isometries. Section 5 ties together the explicit Berger, Coburn and
Lebow representation and other possible analytic representations of
a pair of commuting isometries. Then, in Section 6, we present a
general theory for pairs of commuting isometries and analyze the
defect operators. Concluding remarks, future directions and a close
connection of our consideration with the Sz.-Nagy and Foias
characteristic functions for contractions are discussed in Section
7.

\section{Wold-von Neumann decomposition and commutators}

We begin this section by briefly recalling the construction of the
classical Wold-von Neumann decomposition of isometric operators on
Hilbert spaces. Here our presentation is more algebraic and geared
towards the main theme of the paper. First, recall that an isometry
$V$ on a Hilbert space $\clh$ is said to be \textit{pure}, or a
\textit{shift}, if it has no unitary direct summand, or
equivalently, if $\displaystyle{\lim_{m \raro \infty}} V^{*m} = 0$
in the strong operator topology (see Halmos \cite{H}).

Let $V$ be an isometry on a Hilbert space $\clh$, and let $\clw(V)$
be the \textit{wandering subspace} \cite{H} for $V$, that is,
\[
\clw(V) = \clh \ominus V \clh.
\]

The classical Wold-von Neumann decomposition is as follows:

\begin{theorem}\label{thm-Wold}\textsf{(Wold-von Neumann
decomposition)} Let $V$ be an isometry on a Hilbert space $\clh$.
Then $\clh$ decomposes as a direct sum of $V$-reducing subspaces
$\clh_s(V) = \displaystyle{\mo_{m=0}^\infty} V^m \clw(V)$ and
$\clh_u(V) = \clh \ominus \clh_s(V)$ and
\[
V = \begin{bmatrix} V_s & 0\\ 0 & V_u
\end{bmatrix} \in \clb(\clh_s(V) \oplus \clh_u(V)),
\]
where $V_s = V|_{\clh_s(V)}$ is a shift operator and $V_u = V|_{\clh_u(V)}$ is a unitary operator.
\end{theorem}

We will refer to this decomposition as the \textit{Wold-von Neumann
orthogonal decomposition} of $V$.

Recall that the \textit{Hardy space} $H^2(\D)$ is the
Hilbert space of all analytic functions on the unit disc $\D$ with square
summable Taylor coefficients
(cf. \cite{H}, \cite{RR}). The Hardy space is also a reproducing
kernel Hilbert space corresponding to the Szeg\"{o} kernel
\[
\mathbb{S}(z, w) = (1 - z \bar{w})^{-1} \quad \quad (z, w \in \D).
\]
For any Hilbert space $\cle$, the $\cle$-valued Hardy space with
reproducing kernel
\[
\D \times \D \raro \clb(\cle), \; (z, w) \mapsto \mathbb{S}(z, w)
I_{\cle},
\]
can canonically be identified with the tensor product Hilbert space
$H^2(\D) \otimes \cle$. To simplify the notation, we often identify
$H^2(\D) \otimes \cle$ with the $\cle$-valued Hardy space
$H^2_{\cle}(\D)$. The space of $\clb(\cle)$-valued bounded
holomorphic functions on $\D$ will be denoted by
$H^\infty_{\clb(\cle)}(\D)$.

\NI Let $M_z^{\cle}$ denote the multiplication operator by the
coordinate function $z$ on $H^2_{\cle}(\D)$, that is
\[
(M_z^{\cle} f)(w) = w f(w) \quad \quad \quad (f \in H^2_{\cle}(\D),
w \in \D).
\]
Then $M_z^{\cle}$ is a shift operator and
\[
\clw(M_z^{\cle}) = \cle.
\]

\NI \textsf{To simplify the notation we often omit the superscript
and denote $M_z^{\cle}$ by $M_z$, if $\cle$ is clear from the
context.}

We now proceed to give an analytic description of the Wold-von
Neumann construction.

Let $V$ be an isometry on $\clh$, and let $\clh = \clh_s(V) \oplus
\clh_u(V)$ be the Wold-von Neumann orthogonal decomposition of $V$.
Define
\[
\Pi_V : \clh_s(V) \oplus \clh_u(V) \raro H^2_{\clw(V)}(\D) \oplus
\clh_u(V)
\]
by
\[
\Pi_V (V^m \eta \oplus f) = z^m \eta \oplus f \quad \quad (m \geq 0, \eta \in \clw(V),
f \in \clh_u(V)).
\]
Then $\Pi_V$ is a unitary and
\[
\Pi_V \begin{bmatrix} V_s & 0\\ 0 & V_u
\end{bmatrix} = \begin{bmatrix} M_z^{\clw(V)} & 0\\ 0 & V_u
\end{bmatrix} \Pi_V.
\]
In particular, if $V$ is a shift,
then $\clh_u(V) =\{0\}$ and hence
\[
\Pi_V V = M_z^{\clw(V)} \Pi_V.
\]
Therefore, an isometry $V$ on $\clh$ is a shift operator if and only
if $V$ is unitarily equivalent to $M_z^{\cle}$ on $H^2_{\cle}(\D)$,
where $\dim \cle = \dim \clw(V)$.

\textsf{In the sequel we denote by $(\Pi_V, M_z^{\clw(V)})$, or
simply by  $(\Pi_V, M_z)$, the Wold-von Neumann decomposition of the
pure isometry $V$ in the above sense.}

Let $\cle$ be a Hilbert space, and let $C$ be a bounded linear
operator on $H^2_{\cle}(\D)$. Then $C \in \{M_z\}^{'}$, that is, $C
M_z = M_z C$, if and only if (cf. \cite{NF})
\[
C = M_{\Theta}
\]
for some $\Theta \in H^{\infty}_{\clb(\cle)}(\D)$ and $(M_{\Theta}
f)(w) = \Theta(w) f(w)$ for all $f \in H^2_{\cle}(\D)$ and $w \in
\D$.

Now let $V$ be a pure isometry, and let $C \in \{V\}^{'}$. Let
$(\Pi_V, M_z)$ be the Wold-von Neumann decomposition of $V$, and let
$\clw = \clw(V)$. Since $\Pi_V C \Pi_V^*$ on $H^2_{\clw}(\D)$ is the
representation of $C$ on $\clh$ and $(\Pi_V C \Pi_V^*) M_z = M_z
(\Pi_V C \Pi_V^*)$, it follows that
\[
\Pi_V C \Pi_V^* = M_{\Theta},
\]
for some  $\Theta \in H^{\infty}_{\clb(\clw)}(\D)$. The main result
of this section is the following explicit representation of
$\Theta$.

\begin{theorem}\label{thm-commutator}
Let $V$ be a pure isometry on $\clh$, and let $C$ be a bounded
operator on $\clh$. Let $(\Pi_V, M_z)$ be the Wold-von Neumann
decomposition of $V$. Set $\clw = \clw(V)$, $M = \Pi_V C \Pi^{*}_V$
and let
\[
\Theta(w)= P_{\clw}(I_{\clh} - w V^{*})^{-1}C\mid_{\clw} \quad \quad
(w \in \D).
\]
Then
\[
C V = VC,
\]
if and only if $\Theta \in H^\infty_{\clb(\clw)}(\D)$ and
\[
M = M_{\Theta}.
\]
\end{theorem}
\begin{proof}
Let $h \in \clh$. One can express $h$ as $h =
\displaystyle{\sum_{m=0}^{\infty}} V^m \eta_m$, for some $\eta_m \in
\clw$, $m \geq 0$ (as $\clh = \displaystyle{\mo_{m=0}^{\infty}}
V^{m} {\clw}$). Applying $P_{\clw} V^{*l}$ to both sides and using
the fact that $\clw = \clw(V) = \ker V^*$, we obtain $\eta_l =
P_{\clw} V^{*l} h$ for all $l \geq 0$. This implies, for any $h \in
\clh$,
\begin{equation}\label{eq-F}
h = \sum_{m=0}^{\infty} V^m P_{\clw} V^{*m} h.
\end{equation}
Now let $CV = VC$. Then there exists a bounded analytic function
$\Theta \in H^{\infty}_{\clb(\clw)}(\D)$ such that $\Pi_V C \Pi_V^*
= M_{\Theta}$. For each $w \in \mathbb{D}$ and $\eta \in \clw$ we
have
\[
\begin{split}
\Theta(w) \eta & = (M_{\Theta} \eta)(w)
\\
& = (\Pi_V C \Pi_V^* \eta) (w)
\\
& = (\Pi_V C \eta)(w),
\end{split}
\]
as $\Pi_V^* \eta = \eta$. Since in view of (\ref{eq-F})
\[
C \eta = \sum_{m=0}^{\infty} V^m P_{\clw} V^{*m} C \eta,
\]
it follows that
\[
\begin{split}
\Theta(w) \eta & = (\Pi_V (\sum_{m=0}^{\infty} V^m P_{\clw} V^{*m} C
\eta))(w)
\\
& = (\sum_{m=0}^{\infty} M_z^m (P_{\clw} V^{*m} C \eta))(w)
\\
& = \sum_{m=0}^{\infty} w^m (P_{\clw} V^{*m} C \eta)
\\
& = P_{\clw} (I_{\clh} - w V^*)^{-1} C \eta.
\end{split}
\]
Therefore
\[
\Theta(w) = P_{\clw} (I_{\clh} - w V^*)^{-1} C|_{\clw} \quad \quad
(w \in \D),
\]
as required. Finally, since the sufficient part is trivial, the
proof is complete.
\end{proof}

Note that in the above proof we have used the standard projection
formula (see, for example, Rosenblum and Rovnyak \cite{RR})
$I_{\clh} = \displaystyle{\mbox{SOT}-\sum_{m=0}^\infty} V^m P_{\clw}
V^{*m}$. It may also be observed that $\|w V^*\| = |w| \|V\| < 1$
for all $w \in \D$, and so it follows that the function $\Theta$
defined in Theorem \ref{thm-commutator} is a $\clb(\clw)$-valued
holomorphic function in the unit disc $\D$. However, what is not
guaranteed in general here is that the function $\Theta$ is in
$H^\infty_{\clb(\clw)}(\D)$. The above theorem says that this is so
if and only if $CV = VC$.

\section{Berger, Coburn and Lebow representations}

This section is devoted to a detailed study of Berger, Coburn and
Lebow's representation of pure pairs of commuting isometries. Our
approach is different and yields sharper results, along with new
proofs, in terms of explicit coefficients of the one variable
polynomials associated with the class of pure pairs of commuting
isometries. Before dealing more specifically with pure pairs of
commuting isometries we begin with some general observations about
pairs of commuting isometries.

Let $(V_1, V_2)$ be a pair of commuting isometries on a Hilbert
space $\clh$. \textsf{In the sequel, we will adopt the following
notations}:
\[
V = V_1 V_2,
\]
\[
\clw = \clw(V) = \clw(V_1 V_2) = \clh \ominus V_1 V_2 \clh,
\]
and
\[
\clw_j = \clw(V_j) = \clh \ominus V_j \clh \quad \quad \quad (j = 1,
2).
\]

A pair of commuting isometries $(V_1, V_2)$ on $\clh$ is said to be
\textit{pure} if $V$ is a pure isometry.

The following useful lemma on wandering subspaces for commuting
isometries is simple.

\begin{lemma}\label{lem-U}
Let $(V_1, V_2)$ be a pair of commuting isometries on a Hilbert
space $\clh$. Then
\[
\clw = \clw_1 \oplus V_1 \clw_2 = V_2 \clw_1 \oplus \clw_2,
\]
and the operator $U$ on $\clw$ defined by
\[
U(\eta_1 \oplus V_1 \eta_2) = V_2 \eta_1 \oplus \eta_2,
\]
for $\eta_1 \in \clw_1$ and $\eta_2 \in \clw_2$, is a unitary
operator. Moreover,
\[
P_{\clw} V_i = V_i P_{\clw_j} \quad \quad (i \neq j).
\]
\end{lemma}
\begin{proof}
The first equality follows from
\[
I - V V^* = (I - V_1 V_1^*) \oplus V_1 (I - V_2 V_2^*) V_1^* = V_2
(I - V_1 V_1^*) V_2^* \oplus (I - V_2 V_2^*).
\]
The second part directly follows from the first part, and the last
claim follows from $(I - V V^*) V_i = V_i (I - V_j V_j^*)$ for all
$i \neq j$. This concludes the proof of the lemma.
\end{proof}

Let $(V_1, V_2)$ be a pure pair of commuting isometries on a Hilbert
space $\clh$, and let $(\Pi_V, M_z)$ be the Wold-von Neumann
decomposition of $V$. Since
\[
V V_i = V_i V \quad \quad \quad  (i = 1, 2),
\]
there exist isometric multipliers (that is, inner functions
\cite{NF}) ${\Phi_1}$ and ${\Phi_2}$ in $H^\infty_{\clb(\clw)}(\D)$
such that
\[
\Pi_V V_i = M_{{\Phi_i}} \Pi_V \quad \quad \quad (i = 1, 2).
\]
In other words, $(M_{{\Phi_1}}, M_{{\Phi_2}})$ on $H^2_{\clw}(\D)$
is the representation of $(V_1, V_2)$ on $\clh$. Following Berger,
Coburn and Lebow \cite{BCL}, we say that $(M_{{\Phi_1}},
M_{{\Phi_2}})$ is the \textit{BCL representation} of $(V_1, V_2)$,
or simply the \textit{BCL pair} corresponding to $(V_1, V_2)$.

We now present an explicit description of the BCL pair
$(M_{{\Phi_1}}, M_{{\Phi_2}})$.

\begin{theorem}\label{thm-BCLnew}
Let $(V_1, V_2)$ be a pure pair of commuting isometries on a Hilbert
space $\clh$, and let $(M_{{\Phi_1}}, M_{{\Phi_2}})$ be the BCL
representation of $(V_1, V_2)$. Then
\[
{\Phi_1}(z) = V_1|_{\clw_2} \oplus V_2^*|_{V_2 \clw_1} z, \; {\Phi_2}(z)
= V_2|_{\clw_1} \oplus V_1^*|_{V_1 \clw_2} z,
\]
for all $z \in \D$.
\end{theorem}
\begin{proof}
Let $\eta$ in $\clw = V_2 \clw_1 \oplus \clw_2$, and let $w \in \D$.
Then there exist $\eta_1 \in \clw_1$ and $\eta_2 \in \clw_2$ such
that $\eta = V_2 \eta_1 \oplus \eta_2$. Then $ V_1 \eta = V \eta_1 +
V_1 \eta_2$, and hence
\[
{\Phi_1}(w) \eta  = (M_{{\Phi_1}}\eta)(w) = (\Pi_V V_1 \Pi_V^*
\eta)(w) = (\Pi_V V_1 \eta)(w) = (\Pi_V V \eta_1 + \Pi_V V_1
\eta_2)(w).
\]
This along with the fact that $V_1 \eta_2 \in \clw$ (see Lemma
\ref{lem-U}) gives
\[
\begin{split} {\Phi_1}(w) \eta & = (M_z \Pi_V \eta_1 + V_1 \eta_2) (w)
\\
& = (M_z \eta_1 + V_1 \eta_2) (w)
\\
& = w \eta_1 + V_1 \eta_2\\
& = wV_2^{*}\eta + V_1 \eta_2,
\end{split}
\]
for all $w \in \D$. Therefore
\[
\begin{split}
\Phi_1(z) 
&  = V_1|_{\clw_2} \oplus V_2^*|_{V_2 \clw_1} z,
\end{split}
\]
for all $z \in \D$, as $\clw_2 = Ker(V_2^{*})$. The representation
of $\Phi_2$ follows similarly.
\end{proof}

In the following, we present Berger, Coburn and Lebow's version of
representations of pure pairs of commuting isometries. This yields
an explicit representations of the auxiliary operators $U$ and $P$
(see Section 1). The proof readily follows from Lemma \ref{lem-U}
and Theorem \ref{thm-BCLnew}.

\begin{theorem}\label{thm-BCLold}
Let $(V_1, V_2)$ be a pure pair of commuting isometries on $\clh$.
Then the BCL pair $(M_{{\Phi_1}}, M_{{\Phi_2}})$ corresponding to
$(V_1, V_2)$ is given by
\[
{\Phi_1}(z)= U^* (P_{\clw_2} + zP_{\clw_2}^{\perp}),
\]
and
\[
{\Phi_2}(z) = (P_{\clw_2}^{\perp} + zP_{\clw_2}) U,
\]
where
\[
U=
\begin{bmatrix}
V_2|_{\clw_1} & 0\\
0 & V_1^{*}|_{V_1\clw_2} \end{bmatrix} :
\begin{array}{c}
\clw_1\\
\oplus\\ V_1\clw_2
\end{array}
\raro
\begin{array}{c}
V_2 \clw_1 \\
\oplus\\ \clw_2
\end{array},
\]
is a unitary operator on $\clw$.
\end{theorem}

Therefore, $(V_1, V_2, V_1 V_2)$ on $\clh$ and $(M_{\Phi_1},
M_{\Phi_2}, M_z^{\clw})$ on $H^2_{\clw}(\D)$ are unitarily
equivalent, where $\clw$ is the wandering subspace for $V = V_1
V_2$.

\section{Unitary invariants}

In this short section we present a complete set of joint unitary
invariants for pure pairs of commuting isometries. Recall that two
commuting pairs $(T_1, T_2)$ and $(\tilde{T}_1, \tilde{T}_2)$ on
$\clh$ and $\tilde{\clh}$, respectively, are said to be (jointly)
unitarily equivalent if there exists a unitary operator $U : \clh
\raro \tilde{\clh}$ such that $U T_j = \tilde{T}_j U$ for all $j =
1, 2$.

First we note that, by virtue of Theorem 2.9 of \cite{BDF3}, the
orthogonal projection $P_{\clw_2}$ and the unitary operator $U$ on
$\clw$, as in Theorem \ref{thm-BCLold}, form a complete set of (joint)
unitary invariants of pure pairs of commuting isometries. More
specifically: Let $(V_1, V_2)$ and $(\tilde{V}_1, \tilde{V}_2)$ be
two pure pairs of commuting isometries on $\clh$ and $\tilde{\clh}$,
respectively. Let $\tilde{\clw}_j$ be the wandering subspace for
$\tilde{V}_j$, $j = 1, 2$. Then $(V_1, V_2)$ and $(\tilde{V}_1,
\tilde{V}_2)$ are unitarily equivalent if and only if
\[
( \begin{bmatrix}
V_2|_{\clw_1} & 0\\
0 & V_1^{*}|_{V_1\clw_2} \end{bmatrix}, P_{\clw_2}) \quad \mbox{and}
\quad
(\begin{bmatrix} \tilde{V}_2|_{\tilde{\clw}_1} & 0\\
0 & \tilde{V}_1^{*}|_{\tilde{V}_1 \tilde{{\clw}_2}} \end{bmatrix},
P_{\tilde{\clw}_2})
\]
are unitarily equivalent.

In addition to the above, the following unitary invariants are also
explicit. The proof is an easy consequence of Theorem
\ref{thm-BCLnew}. Here we will make use of the identifications of
$A$ on $H^2_{\clw}(\D)$ and $A M_z$ on $H^2_{\clw}(\D)$ with
$I_{H^2(\D)} \otimes A$ on $H^2(\D) \otimes \clw$ and $M_z \otimes
A$ on $H^2(\D) \otimes \clw$, respectively, where $A \in \clb(\clw)$
(see Section 2).

\begin{theorem}
Let $(V_1, V_2)$ and $(\tilde{V}_1, \tilde{V}_2)$ be two pure pairs
of commuting isometries on $\clh$ and $\tilde{\clh}$, respectively.
Then $(V_1, V_2)$ and $(\tilde{V}_1, \tilde{V}_2)$ are unitarily
equivalent if and only if $(V_1|_{\clw_2}, V_2^*|_{V_2 \clw_1})$ and
$(\tilde{V}_1|_{\tilde{\clw_2}}, \tilde{V}_2^*|_{\tilde{V}_2
\tilde{\clw}_1})$ are unitarily equivalent.
\end{theorem}
\begin{proof}
Let $(M_{\Phi_1}, M_{\Phi_2})$ and $(M_{\tilde{\Phi}_1},
M_{\tilde{\Phi}_2})$ be the BCL pairs corresponding to $(V_1, V_2)$
and $(\tilde{V}_1, \tilde{V}_2)$, respectively, as in Theorem
\ref{thm-BCLnew}. Let $C_1 = V_1|_{\clw_2}$ and $C_2 = V_2^*|_{V_2
\clw_1}$ be the coefficients of $\Phi_1$. Similarly, let
$\tilde{C}_1$ and $\tilde{C}_2$ be the coefficients of
$\tilde{\Phi}_1$.

\NI Now let $Z : \clw \raro \tilde{\clw}$ be a unitary such that $Z
C_j = \tilde{C}_j Z$, $j = 1, 2$. Then
\[
\begin{split}
M_{\Phi_1} & = I_{H^2(\D)} \otimes C_1 + M_z \otimes C_2
\\
& = I_{H^2(\D)} \otimes Z^* \tilde{C}_1 Z + M_z \otimes Z^*
\tilde{C}_2 Z
\\
& = (I_{H^2(\D)} \otimes Z^*) (I_{H^2(\D)} \otimes \tilde{C}_1 + M_z
\otimes \tilde{C}_2) (I_{H^2(\D)} \otimes Z)
\\
& = (I_{H^2(\D)} \otimes Z^*) M_{\tilde{\Phi}_1} (I_{H^2(\D)}
\otimes Z).
\end{split}
\]
Because $M_{\Phi_2} = M_z M_{{\Phi}_1}^*$ and $M_{\tilde{\Phi}_2} =
M_z M_{\tilde{\Phi}_1}^*$, it follows that $(M_{\Phi_1},
M_{\Phi_2})$ and $(M_{\tilde{\Phi}_1}, M_{\tilde{\Phi}_2})$ are
unitarily equivalent, that is, $(V_1, V_2)$ and $(\tilde{V}_1,
\tilde{V}_2)$ are unitarily equivalent.

\NI To prove the necessary part, let $(M_{\Phi_1}, M_{\Phi_2})$ and
$(M_{\tilde{\Phi}_1}, M_{\tilde{\Phi}_2})$ are unitarily equivalent.
Then there exists a unitary operator $X : H^2_{\clw}(\D) \raro
H^2_{\tilde{\clw}}(\D)$ \cite{RR} such that
\[
X M_{\Phi_j} = M_{\tilde{\Phi}_j} X \quad \quad (j = 1, 2).
\]
Since
\[
X M_z^{\clw}  = X M_{\Phi_1} M_{\Phi_2} = M_{\tilde{\Phi}_1} X X^*
M_{\tilde{\Phi}_2} X = M_{\tilde{\Phi}_1} M_{\tilde{\Phi}_2} X =
M_z^{\tilde{\clw}} X,
\]
there exists a unitary operator $Z : \clw \raro \tilde{\clw}$ such
that
\[
X = I_{H^2(\D)} \otimes Z.
\]
This and $X M_{\Phi_1} = M_{\tilde{\Phi}_1} X$ implies that
\[
(I_{H^2(\D)} \otimes Z) (I_{H^2(\D)} \otimes C_1 + M_z \otimes C_2)
= (I_{H^2(\D)} \otimes \tilde{C}_1 + M_z \otimes \tilde{C}_2)
(I_{H^2(\D)} \otimes Z).
\]
Hence $(C_1, C_2)$ and $(\tilde{C}_1, \tilde{C}_2)$ are unitarily
equivalent. This completes the proof of the theorem.
\end{proof}

Observe that the set of joint unitary invariants $\{V_1|_{\clw_2},
V_2^*|_{V_2 \clw_1}\}$, as above, is associated with the
coefficients of $\Phi_1$ of the BCL pair $(M_{\Phi_1}, M_{\Phi_2})$
corresponding to $(V_1, V_2)$. Clearly, by duality, a similar
statement holds for the coefficients of $\Phi_2$ as well:
$\{V_2|_{\clw_1}, V_1^*|_{V_1 \clw_2}\}$ is a complete set of joint
unitary invariants for pure pairs of commuting isometries.

\section{Pure isometries}

In this section we will analyze pairs of commuting isometries $(V_1,
V_2)$ such that either $V_1$ or $V_2$ is a pure isometry, or both
$V_1$ and $V_2$ are pure isometries. We begin with a concrete
example which illustrates this particular class and also exhibits its
complex structure.

Recall that the Hardy space $H^2(\D^2)$ over the bidisc $\D^2$ is
the Hilbert space of all analytic
functions on the bidisc $\D^2$ with square summable Taylor coefficients
(see Rudin \cite{RUD}). Let $M_{z_j}$ on $H^2(\D^2)$ be the
multiplication operator by the coordinate function $z_j$, $j = 1,
2$. Note that $(M_{z_1}, M_{z_2})$ on $H^2(\D^2)$ can be identified
with $(M_z \otimes I_{H^2(\D)}, I_{H^2(\D)} \otimes M_z)$ on
$H^2(\D) \otimes H^2(\D)$, and consequently, $(M_{z_1}, M_{z_2})$ on
$H^2(\D^2)$ is a pair of \textit{doubly commuting} (that is,
$M_{z_1}^* M_{z_2} = M_{z_2} M_{z_1}^*$) pure isometries.

\NI Now let $\cls$ be a joint $(M_{z_1}, M_{z_2})$-invariant closed
subspace of $H^2(\D^2)$, that is, $M_{z_j} \cls \subseteq \cls$. Set
\[
V_j = M_{z_j}|_{\cls} \quad \quad (j = 1, 2).
\]
It follows immediately that $V_j$ is a pure isometry and $V_1 V_2 =
V_2 V_1$, and hence $(V_1, V_2)$ is a pair of commuting pure
isometries on $\cls$.

\NI If we assume, in addition, that $(V_1, V_2)$ is doubly commuting
(that is, $V_1^* V_2 = V_2 V_1^*$), then it follows that $(V_1,
V_2)$ on $\cls$ and $(M_{z_1}, M_{z_2})$ on $H^2(\D^2)$ are
unitarily equivalent. See Slocinski \cite{SLO} for more details. In
general, however, the classification of pairs of commuting
isometries, up to unitary equivalence, is complicated and very
little seems to be known. For instance, see Rudin \cite{RUD} for a
list of pathological examples (also see Qin and Yang \cite{QY}).

We now turn our attention to the general problem. Let $(V_1, V_2)$
be a pair of commuting isometries on $\clh$, and let $V_1$ be a pure
isometry. Then, in particular, $V = V_1 V_2$ is a pure isometry, and
hence $(V_1, V_2)$ is a pure pair of commuting isometries. Since
$V_1 V_2 = V_2 V_1$, by Theorem \ref{thm-commutator}, it follows
that
\begin{equation}\label{eqn-pureV1}
\Pi_{V_1} V_2 = M_{\Theta_{V_2}} \Pi_{V_1},
\end{equation}
where $\Theta_{V_2} \in H^\infty_{\clb(\clw_1)}(\D)$ is an inner
multiplier and
\begin{equation}\label{eqn-theta12}
\Theta_{V_2}(z) = P_{\clw_1} (I_{\clh} - z V_1^*)^{-1} V_2|_{\clw_1}
\quad \quad (z \in \D).
\end{equation}
Let $(M_{{\Phi_1}}, M_{{\Phi_2}})$ be the BCL pair (see Theorem
\ref{thm-BCLold}) corresponding to $(V_1, V_2)$, that is, $\Pi_V V_i
= M_{{\Phi_i}} \Pi_V$ for all $i = 1, 2$. Set
\[
\TP_1 = \Pi_{V_1} \Pi_V^*.
\]
Then $\TP_1 : H^2_{\clw}(\D) \raro H^2_{\clw_1}(\D)$ is a unitary
operator such that $\TP_1 M_{{\Phi_1}}  = M_z^{\clw_1} \TP_1$ and
$\TP_1 M_{{\Phi_2}} = M_{\Theta_{V_2}} \TP_1$. Therefore, we have
the following commutative diagram:
\[
\xymatrix{
\mathcal H \ar[r]^{ \Pi_V} \ar[dr]_{\Pi_{V_1}} & H^2_{\clw}(\D) \ar[d]^{\TP_1} \\
 & H^2_{\clw_1}(\D)
}
\]
where $(M_{{\Phi_1}}, M_{{\Phi_2}})$ on $H^2_{\clw}(\D)$ and
$(M_z^{\clw_1}, M_{\Theta_{V_2}})$ on $H^2_{\clw_1}(\D)$ are the
representations of $(V_1, V_2)$ on $\clh$.

We now proceed to settle the non-trivial part of this consideration:
An analytic description of the unitary map $\TP_1$. To this end,
observe first that since $\Pi_{V_1} V_1 = M_z^{\clw_1} \Pi_{V_1}$,
(\ref{eqn-pureV1}) gives
\[
\Pi_{V_1} V = M_z^{\clw_1} M_{\Theta_{V_2}} \Pi_{V_1}.
\]
Then
\[
\TP_1 M_z^{\clw} = \Pi_{V_1} V \Pi_V^* = M_z^{\clw_1}
M_{\Theta_{V_2}} \Pi_{V_1} \Pi_V^*,
\]
that is,
\begin{equation}\label{eqn-TPw}
\TP_1 M_z^{\clw} =  (M_z^{\clw_1} M_{\Theta_{V_2}}) \TP_1.
\end{equation}
Let $\eta \in \clw$. By Equation (\ref{eq-F}) we can write $\eta =
\displaystyle{\sum_{m=0}^\infty} V_1^m P_{\clw_1} V_1^{*m} \eta$.
Therefore
\[
\begin{split}
(\Pi_{V_1} \eta)(w) & = (\sum_{m=0}^\infty \Pi_{V_1} V_1^m
P_{\clw_1} V_1^{*m} \eta)(w)
\\
& = (\sum_{m=0}^\infty M_{z}^m P_{\clw_1} V_1^{*m} \eta)(w)
\\
& = \sum_{m=0}^\infty w^m (P_{\clw_1} V_1^{*m} \eta),
\end{split}
\]
which yields
\[
\TP_1 \eta = \Pi_{V_1} \Pi_V^* \eta  = \Pi_{V_1} \eta =
\sum_{m=0}^\infty z^m (P_{\clw_1} V_1^{*m} \eta),
\]
that is
\[
\TP_1 \eta = P_{\clw_1} [I_{\clh} + z (I_{\clh} - z V_1^*)^{-1}
V_1^*] \eta,
\]
for all $\eta \in \clw$. It now follows from (\ref{eqn-TPw}) that
\[
\TP_1 (z^m \eta) = (z \Theta_{V_2}(z))^m P_{\clw_1} [I_{\clh} + z
(I_{\clh} - z V_1^*)^{-1} V_1^*] \eta,
\]
for all $m \geq 0$, and so, by $\mathbb{S}(\cdot, w) \eta =
\displaystyle{\sum_{m=0}^\infty} z^m \bar{w}^m \eta$, it follows
that
\[
\begin{split}
\TP_1(\mathbb{S}(\cdot, w) \eta) & = \TP_1 (\sum_{m=0}^\infty z^m
\bar{w}^m \eta)
\\
& = (I_{\clw_1} - \bar{w} z \Theta_{V_2}(z))^{-1} P_{\clw_1}
[I_{\clh} + z (I_{\clh} - z V_1^*)^{-1} V_1^*] \eta,
\end{split}
\]
for all $w \in \D$ and $\eta \in \clw$. Finally, from $\TP_1^*
M_z^{\clw_1} = M_{{\Phi_1}} \TP_1^*$ and $\TP_1^* \eta_1 = \eta_1$
for all $\eta_1 \in \clw_1$, it follows that $\TP_1^* (z^m \eta_1) =
M_{{\Phi_1}}^m \eta_1$ for all $m \geq 0$, and hence
\[
\TP_1^* (\mathbb{S}(\cdot, w) \eta_1) = (I_{\clw} - {\Phi_1}(z)
\bar{w})^{-1} \eta_1,
\]
for all $w \in \D$ and $\eta_1 \in \clw_1$.

We summarize the above observations in the following theorem.

\begin{theorem}\label{thm-one pure}
Let $(V_1, V_2)$ be a pair of commuting isometries on $\clh$. Let
$i, j \in \{ 1, 2 \}$ and $i \neq j$. If $V_i$ is a pure isometry,
then
\[
\TP_i = \Pi_{V_i} \Pi_V^* \in \clb(H^2_{\clw}(\D),
H^2_{\clw_i}(\D)),
\]
is a unitary operator,
\[
\TP_i M_z^{\clw} = M_{z \Theta_{V_j}} \TP_i, \; \TP_i^*
M_{^z}^{\clw_i} = M_{{\Phi_i}} \TP_i^*,
\]
and
\[
\TP_i(\mathbb{S}(\cdot, w) \eta) = (I_{\clw_i} - \bar{w} z
\Theta_{V_j}(z))^{-1} P_{\clw_i} [I_{\clh} + z (I - z V_i^*)^{-1}
V_i^*] \eta,
\]
for all $w \in \D$ and $\eta \in \clw$, where
\[
\Theta_{V_j}(z) = P_{\clw_i}(I_{\clh} - z V_i^*)^{-1} V_j|_{\clw_i}
\]
for all $z \in \D$. Moreover
\[
\TP_i^* (\mathbb{S}(\cdot, w) \eta_i) = (I_{\clw} - {\Phi_i}(z)
\bar{w})^{-1} \eta_i,
\]
for all $w \in \D$ and $\eta_i \in \clw_i$.
\end{theorem}

Note that the inner multipliers $\Theta_{V_i} \in
H^\infty_{\clb(\clw_j)}(\D)$ above satisfy the following
equalities:
\[
\Pi_{V_j} V_i = M_{\Theta_{V_i}} \Pi_{V_j}.
\]

Now let $(V_1, V_2)$ be a pair of commuting isometries such that
both $V_1$ and $V_2$ are pure isometries. The above result leads to
an analytic representation of such pairs.

\begin{corollary}\label{cor-one pure}
Let $(V_1, V_2)$ be a pair of
commuting pure isometries on a Hilbert space $\clh$. If
$(M_{{\Phi_1}}, M_{{\Phi_2}})$ is the BCL representation
corresponding to $(V_1, V_2)$, then $M_{{\Phi_1}}$ and
$M_{{\Phi_2}}$ are pure isometries,
\[
\TP_1 M_{{\Phi_2}} = M_{\Theta_{V_2}} \TP_1, \; \TP_2 M_{{\Phi_1}} =
M_{\Theta_{V_1}} \TP_2,
\]
$\TP = \TP_2 \TP_1^* : H^2_{\clw_1}(\D) \raro H^2_{\clw_2}(\D)$ is a
unitary operator, and
\[
\TP M_z^{\clw_1} = M_{\Theta_{V_1}} \TP \; \text{and} \; \TP
M_{\Theta_{V_2}} = M_z^{\clw_2} \TP.
\]
Moreover, for each $w \in \D$ and $\eta_j \in \clw_j$, $j =1, 2$,
\[
\TP (\mathbb{S}(\cdot, w) \eta_1) = (I_{\clw_2} - \bar{w}
\Theta_{V_1}(z) )^{-1} P_{\clw_2} ( I_{\clh} - z V_2^*)^{-1} \eta_1,
\]
and
\[
\TP^* (\mathbb{S}(\cdot, w) \eta_2) = (I_{\clw_1} - \bar{w}
\Theta_{V_2}(z) )^{-1} P_{\clw_1} ( I_{\clh} - z V_1^*)^{-1} \eta_2.
\]
\end{corollary}

\begin{proof}
A repeated application of Theorem \ref{thm-one pure} yields
\[
\begin{split}
\TP_1 M_{{\Phi_2}} & = \TP_1 M_{{\Phi_1}}^* (M_{{\Phi_1}}
M_{{\Phi_2}})
\\
& = \TP_1 M_{{\Phi_1}}^* M_z^{\clw}
\\
& = (M_z^{\clw_1})^* \TP_1 M_z^{\clw}
\\
& = (M_z^{\clw_1})^* M_{z \Theta_{V_2}} \TP_1,
\end{split}
\]
that is, $\TP_1 M_{{\Phi_2}} = M_{\Theta_{V_2}} \TP_1$ and similarly
$\TP_2 M_{{\Phi_1}} = M_{\Theta_{V_1}} \TP_2$. For $\eta_1 \in
\clw_1$, we have $\Pi_{V_2} \eta_1 = P_{\clw_2} ( I_{\clh} - z
V_2^*)^{-1} \eta_1$. Since $\tilde{\Pi}_1^* \eta_1 = \eta_1$ and
$\Pi_V^* \eta_1 = \eta_1$, it follows that
\[
\TP \eta_1 = \TP_2 \eta_1 = \Pi_{V_2} \Pi_V^* \eta_1 = \Pi_{V_2}
\eta_1,
\]
that is $\TP \eta_1 = P_{\clw_2} ( I_{\clh} - z V_2^*)^{-1} \eta_1$.
Now using the identity $\TP (z \eta_1) = M_{\Theta_{V_1}} \TP
\eta_1$, we have
\[
\TP (z^m \eta_1) = \Theta_{V_1}(z)^m  P_{\clw_2} ( I_{\clh} - z
V_2^*)^{-1} \eta_1,
\]
for all $m \geq 0$ and $\eta_1 \in \clw_1$. Finally
$\mathbb{S}(\cdot, w) \eta_1 = \displaystyle{\sum_{m=0}^\infty}
\bar{w}^m z^m \eta_1$ gives
\[
\TP (\mathbb{S}(\cdot, w) \eta_1) = (I_{\clw_2} - \bar{w}
\Theta_{V_1}(z) )^{-1} P_{\clw_2} ( I_{\clh} - z V_2^*)^{-1} \eta_1.
\]
The final equality of the corollary follows from the equality
\[
\TP^* (z^m \eta_2) = \Theta_{V_2}(z)^m (\TP^* \eta_2) =
\Theta_{V_2}(z)^m P_{\clw_1}(I_{\clh} - z V_1^*)^{-1} \eta_2,
\]
for all $m \geq 0$ and $\eta_2 \in \clw_2$. This concludes the
proof.
\end{proof}

In the final section, we will connect the analytic descriptions of
$\tilde{\Pi}_1$ and $\tilde{\Pi}_2$ as in Theorem \ref{thm-one pure}
with the classical notion of the Sz.-Nagy and Foias characteristic
functions of contractions on Hilbert spaces \cite{NF}.

\section{Defect Operators}

Throughout this section, we will mostly work on general (not
necessarily pure) pairs of commuting isometries. Let $(V_1, V_2)$ be
a pair of commuting isometries on a Hilbert space $\clh$. The
\textit{defect operator} $C(V_1,V_2)$ of $(V_1, V_2)$ (cf.
\cite{HQY}) is defined as the self-adjoint
operator
\[
C(V_1,V_2) = I - V_1 {V_1}^* - V_2 {V_2}^* + V_1 V_2 {V_1}^*
{V_2}^*.
\]

Recall from Section 3 that given a pair of commuting isometries
$(V_1, V_2)$, we write $V = V_1 V_2$, and denote by
\[
\clw_j = \clw(V_j) =  \ker V_j^* = \clh \ominus V_j \clh,
\]
the wandering subspace for $V_j$, $j = 1, 2$. The wandering subspace
for $V$ is denoted by $\clw$. Finally, we recall that (see Lemma
\ref{lem-U}) $\clw= \clw_1 \oplus V_1\clw_2 = V_2\clw_1 \oplus
\clw_2$. This readily implies
\begin{equation}\label{eqn-W}
P_{\clw} = P_{\clw_1} \oplus P_{V_1 \clw_2} = P_{V_2\clw_1} \oplus
P_{\clw_2}.
\end{equation}

The following lemma is well known to the experts, but for the sake
of completeness we provide a proof of the statement.

\begin{lemma}\label{lem-Vand12}
Let $(V_1, V_2)$ be a commuting pair of isometries on $\clh$. Then
$\clh_s(V)$ and $\clh_u(V)$ are $V_j$-reducing subspaces,
\[
\clh_s(V_j) \subseteq \clh_s(V), \; \text{and} \;\clh_u(V_j)
\supseteq \clh_u(V),
\]
for all $j = 1, 2$.
\end{lemma}
\begin{proof}
For the first part we only need to prove that $\clh_s(V)$ is a
$V_1$-reducing subspace. Note that since (see Lemma \ref{lem-U})
$V_1 \clw \subseteq \clw \oplus V \clw$, it follows that
\[
V_1 V^m \clw \subseteq V^m (\clw \oplus V \clw) \subseteq \clh_s(V),
\]
for all $m \geq 0$. This clearly implies that $V_1 \clh_s(V)
\subseteq \clh_s(V)$. On the other hand, since $V_1^* \clw = \clw_2
\subseteq \clw$ and
\[
V_1^* V^m \clw = V^{m-1}(V_2 \clw) \subseteq V^{m-1} (\clw \oplus V
\clw),
\]
it follows that $V_1^* \clh_s(V) \subseteq \clh_s(V)$. To prove the
second part of the statement, it is enough to observe that
\[
V^m \clh = V_1^m (V_2^m \clh) = V_2^m (V_1^m \clh) \subseteq V_1^m
\clh, V_2^m \clh,
\]
for all $m \geq 0$, and as $n \raro \infty$
\[
V_1^{*n} h \raro 0, \; \mbox{or} \;V_2^{*n} h \raro 0 \Rightarrow
V^{*n} h \raro 0,
\]
for any $h \in \clh$. This concludes the proof of the lemma.
\end{proof}

The following characterizations of doubly commuting isometries will
prove important in the sequel.

\begin{lemma} \label{lemma21}
Let $(V_1, V_2) $ be a pair of commuting isometries on a Hilbert
space $\clh$.  Then the following are equivalent:

\NI(i) $(V_1, V_2) $ is doubly commuting.

\NI (ii) $V_2 \clw_1 \subseteq \clw_1$.

\NI (iii) $V_1\clw_2 \subseteq \clw_2$.
\end{lemma}

\begin{proof}
Since
(i) implies (ii) and (iii), by symmetry we only need to show that (ii)
implies (i).
Let $V_2\clw_1 \subseteq
\clw_1 $. Let $\clh = \clh_s(V) \oplus \clh_u(V)$ be the Wold-von
Neumann orthogonal decomposition of $V$ (see Theorem
\ref{thm-Wold}). Then $\clh_s(V)$ and $\clh_u(V)$ are joint $(V_1,
V_2)$-reducing subspaces, and the pair $(V_1|_{\clh_u(V)},
V_2|_{\clh_u(V)})$ on $\clh_u$ is doubly commuting, because
$V_j|_{\clh_u(V)}$, $j = 1, 2$, are unitary operators, by Lemma
\ref{lem-Vand12}. Now it only remains to prove that $V_1^* V_2 = V_2
V_1^*$ on $\clh_s(V)$. Since
\[
(V_1^* V_2  - V_2V_1^*)V^m = V_1^*V^m V_2  - V_2 V_1^*V^m = V^{m-1}
V_2^2 - V_2^2 V^{m-1} =0,
\]
it follows that $V_1^* V_2  - V_2V_1^* = 0$ on $V^m \clw$ for all $m
\geq 1$. In order to complete the proof we must show that $V_1^* V_2
= V_2V_1^*$ on $\clw$. To this end, let $\eta = \eta_1 \oplus
V_1\eta_2 \in \clw$ for some $\eta_1 \in \clw_1$ and $\eta_2 \in
\clw_2$. Then
\[
V_1^*V_2(\eta_1 \oplus V_1\eta_2) = V_1^*V_2\eta_1 +
V_1^*V_2V_1\eta_2 = V_2\eta_2,
\]
as $V_2\clw_1 \subseteq \clw_1$, and on the other hand
\[
V_2V_1^*(\eta_1 \oplus V_1\eta_2) = V_2V_1^*\eta_1 +
V_2V_1^*V_1\eta_2 = V_2\eta_2.
\]
This completes the proof.
\end{proof}

The key of our geometric approach is the following simple
representation of defect operators.

\begin{lemma}\label{lem-C}
\[
C(V_1, V_2) = P_{\clw_1} - P_{V_2\clw_1} = P_{\clw_2} -
P_{V_1\clw_2}.
\]
\end{lemma}
\begin{proof}
The result readily follows from (\ref{eqn-W}) and
\[
\begin{split}
C(V_1, V_2) & = (I- V_1 {V_1}^*) +(I - V_2 {V_2}^*) - (I - VV^*)\\
& = P_{\clw_1} + P_{\clw_2} - P_{\clw}.
\end{split}
\]
\end{proof}

The final ingredient to our analysis is the fringe operator $F_2$.
The notion of fringe operators plays a significant role in the study
of joint shift-invariant closed subspaces of the Hardy space over
$\D^2$ (see the discussion at the beginning of Section 5). Given a
pair of commuting isometries $(V_1, V_2)$ on $\clh$, the
\textit{fringe operators} $F_1 \in \clb(\clw_2)$ and $F_2 \in
\clb(\clw_1)$ are defined by
\[
F_j = P_{\clw_i} V_j|_{\clw_i} \quad \quad \quad (i \neq j).
\]
Of particular interest to us are the isometric fringe operators. Note
that
\[
F_2^* F_2 = P_{\clw_1} V_2^* P_{\clw_1} V_2|_{\clw_1}.
\]

\begin{lemma}\label{lem-F1}
The fringe operator $F_2$ on $\clw_1$ is an isometry if and only if
$V_2 \clw_1 \subseteq \clw_1$.
\end{lemma}
\begin{proof}
As $I_{\clw_1} - F_2^* F_2 = I_{\clw_1} - P_{\clw_1} V_2^*
P_{\clw_1} V_2|_{\clw_1}$, (\ref{eqn-W}) implies that
\[
\begin{split}
I_{\clw_1} - F_2^* F_2 & = P_{\clw_1} V_2^* P_{V_1 \clw_2}
V_2|_{\clw_1}.
\end{split}
\]
Then $F_2^* F_2 = I_{\clw_1}$ if and only if $P_{V_1 \clw_2}
V_2|_{\clw_1} = 0$, or, equivalently, if and only if $V_2 \clw_1
\perp V_1 \clw_2 = \clw_1^{\perp}$, by Lemma \ref{lem-U}. This
completes the proof.
\end{proof}

Therefore, the fringe operator $F_2$ is an isometry if and only if
the pair $(V_1, V_2)$ is doubly commuting.

We are now ready to formulate a generalization of Theorem 3.4 in
\cite{HQY} by He, Qin and Yang. Here we do not assume that $(V_1,
V_2)$ is pure.

\begin{theorem}\label{thm-defect1}
Let $(V_1, V_2)$ be a pair of commuting isometries on $\clh$. Then
the following are equivalent:

\NI (a) $C(V_1,V_2) \geq 0$.

\NI (b) $V_2 \clw_1 \subseteq \clw_1$.

\NI (c) $(V_1, V_2)$ is doubly commuting.

\NI (d) $C(V_1,V_2)$ is a projection.

\NI (e) The fringe operator $F_2$ is an isometry.
\end{theorem}
\begin{proof}
The equivalences of (a) and (b), (b) and (c), and (b) and (e) are
given in Lemma \ref{lem-C}, Lemma \ref{lemma21} and Lemma
\ref{lem-F1}, respectively. The implication (c) implies (d) follows
from
\[
C(V_1, V_2) = P_{\clw_1} P_{\clw_2} = P_{\clw_2} P_{\clw_1}.
\]
Clearly (d) implies (a). This completes the proof.
\end{proof}

We now prove that for a large class of pairs of commuting isometries
negative defect operator always implies the zero defect operator.

\begin{theorem}\label{thm--C}
Let $(V_1, V_2)$ be a pair of commuting isometries on $\clh$.
Suppose that $V_1$ or $V_2$ is pure. Then $C(V_1,V_2) \leq 0$ if and
only if $ C(V_1,V_2) =0$.
\end{theorem}

\begin{proof}
With out loss of generality assume that $V_2$ is pure. If $C(V_1,
V_2) \leq 0$, then by Lemma \ref{lem-C}, we have $P_{\clw_1} \leq
P_{V_2 \clw_1}$, or, equivalently
\[
\clw_1 \subseteq V_2 \clw_1,
\]
and hence
\[
\clw_1 \subseteq {V_2}^m \clw_1 \subseteq {V_2}^m \clh ,
\]
for all $m \geq 0$. Therefore
\[
\clw_1 = \mathop{\cap}_{m=0}^\infty {V_2}^m \clw_1 \subseteq
\mathop{\cap}_{m=0}^\infty{V_2}^m \clh = \{ 0 \},
\]
as $V_2 $ is pure. Hence $\clw_1 = \{ 0 \}$ and ${V_2} \clw_1 = \{ 0
\}$. This gives $C(V_1, V_2) = P_{\clw_1} - P_{V_2\clw_1} = 0$.
\end{proof}

The same conclusion holds if we allow $\mbox{dim~} \clw_j < \infty$
for some $j \in \{1, 2\}$.

\begin{theorem}\label{thm--C}
Let $(V_1, V_2)$ be a pair of commuting isometries on $\clh$.
Suppose that $\mbox{dim~} \clw_j < \infty$ for some $j \in \{1,
2\}$. Then $C(V_1,V_2) \leq 0$ if and only if $ C(V_1,V_2) =0$.
\end{theorem}

\begin{proof}
We may suppose that $\mbox{dim~} \clw_1 < \infty$. Let $ C(V_1,V_2)
\leq 0$. Since $\clw_1 \subseteq V_2 \clw_1$ and $V_2$ is an
isometry, it follows that
\[
\clw_1 = V_2 \clw_1.
\]
Hence $C(V_1, V_2) = P_{\clw_1} - P_{V_2\clw_1} = 0$. This completes
the prove.
\end{proof}

The same conclusion also holds for positive defect operators.

\section{Concluding Remarks}

As pointed out in the introduction, a general theory for pairs of
commuting isometries is mostly unknown and unexplored (however, see
Popovici \cite{P}). In comparison, we would like to add that a great
deal is known about the structure of pairs (and even of $n$-tuples)
of commuting isometries with finite rank defect operators (see
\cite{BKS}, \cite{BKPS-1}, \cite{BKPS-2}). A complete classification
result is also known for $n$-tuples of doubly commuting isometries
(cf. \cite{GS}, \cite{SLO}, \cite{JS}). It is now natural to ask
whether the present results for pure pairs of commuting isometries
can be extended to arbitrary pairs of commuting isometries (see
\cite{D-Acta}, \cite{GG} and \cite{GS} for closely related results).
Another relevant question is to analyze the joint shift invariant
subspaces of the Hardy space over the unit bidisc \cite{ACD} from
our analytic and geometric point of views. More detailed discussion
on these issues will be given in forthcoming papers.

Also we point out that some of the results of this paper can be
extended to $n$-tuples of commuting isometries and will be discussed
in a future paper.

We conclude this paper by inspecting a connection between the
Sz.-Nagy and Foias characteristic functions of contractions on
Hilbert spaces \cite{NF} and the analytic representations of
$\tilde{\Pi}_1$ and $\tilde{\Pi}_2$ as described in Theorem
\ref{thm-one pure}.

\NI Let $T$ be a contraction on a Hilbert space $\clh$. The
\textit{defect operators} of $T$, denoted by $D_{T^*}$ and $D_T$,
are defined by
\[
D_{T^*} = ( I - TT^*)^{1/2}, \;\; D_{T} = ( I - T^* T)^{1/2}.
\]
The defect spaces, denoted by $\cld_{T^*}$ and $\cld_T$, are the
closure of the ranges of $D_{T^*}$ and $D_T$, respectively. The
\textit{characteristic function} \cite{NF} of the contraction $T$ is
defined by
\[
\theta_T(z) = [- T + z D_{T^*} (I - zT^*)^{-1} D_T]|_{\cld_T} \quad
\quad (z \in \D).
\]
It follows that $\theta_T \in H^\infty_{\clb(\cld_T,
\cld_{T^*})}(\D)$ \cite{NF}. The characteristic function is a
complete unitary invariant for the class of completely non-unitary
contractions. This function is also closely related to the
Beurling-Lax-Halmos inner functions for shift invariant subspaces of
vector-valued Hardy spaces. For a more detailed discussion of the
theory and applications of characteristic functions we refer to the
monograph by Sz.-Nagy and Foias \cite{NF}.

\NI Now let us return to the study of pairs of commuting isometries.
Let $(V_1, V_2)$ be a pair of commuting isometries on $\clh$. We
compute
\[
\begin{split}
P_{\clw_1}[I_{\clh} + z (I_{\clh} - z V_1^*)^{-1} V_1^*]|_{\clw} & =
[P_{\clw_1} + z P_{\clw_1} (I_{\clh} - z V_1^*)^{-1} V_1^*]|_{\clw}
\\
& = [I_{\clh} - V_1 V_1^* + z P_{\clw_1} (I_{\clh} - z V_1^*)^{-1}
V_1^*]|_{\clw}
\\
& = I_{\clw} + [ - V_1 + z P_{\clw_1} (I_{\clh} - z
V_1^*)^{-1}]V_1^*|_{\clw}.
\end{split}
\]
Since $V_1^* \clw = \clw_2$, it follows that
\[
[ - V_1 + z P_{\clw_1} (I_{\clh} - z V_1^*)^{-1}]V_1^*|_{\clw} = [-
V_1 + z D_{V^*_1} (I_{\clh} -z V_1^*)^{-1}
D_{V_2^*}]|_{\cld_{V_2^*}} (V_1^*|_{\clw}).
\]
Therefore, setting
\begin{equation}\label{eqn-theta12}
\theta_{V_1, V_2}(z) = [- V_1 + z D_{V^*_1} (I_{\clh} -z V_1^*)^{-1}
D_{V_2^*}]|_{\cld_{V_2^*}},
\end{equation}
for $z \in \D$, we have
\[
P_{\clw_1}[I_{\clh} + z (I_{\clh} - z V_1^*)^{-1} V_1^*]|_{\clw} =
I_{\clw} + \theta_{V_1, V_2}(z) V_1^*|_{\clw},
\]
for all $z \in \D$. Therefore, if $V_1$ is a pure isometry, then the
formula for $\tilde{\Pi}_1$ in Theorem \ref{thm-one pure}(i) can be
expressed as
\[
\tilde{\Pi}_1(\mathbb{S}(\cdot, w) \eta) = (I_{\clw_1} - \bar{w}
\Theta_{V_2}(z))^{-1} P_{\clw_1} [I_{\clw} + \theta_{V_1, V_2}(z)
V_1^*|_{\clw}] \eta.
\]
for all $w \in \D$ and $\eta \in \clw$. Similarly, if $V_2$ is a
pure isometry, then the formula for $\tilde{\Pi}_2$ in Theorem
\ref{thm-one pure} (ii) can be expressed as
\[
\tilde{\Pi}_2(\mathbb{S}(\cdot, w) \eta) = (I_{\clw_2} - \bar{w}
\Theta_{V_1}(z))^{-1} P_{\clw_2} [I_{\clw} + \theta_{V_2, V_1}(z)
V_2^*|_{\clw}] \eta,
\]
for all $w \in \D$ and $\eta \in \clw$, where
\begin{equation}\label{eqn-theta21}
\theta_{V_2, V_1}(z) =  [- V_2 + z D_{V^*_2} (I_{\clh} -z
V_2^*)^{-1} D_{V_1^*}]|_{\cld_{V_1^*}},
\end{equation}
for all $z \in \D$.

\NI It is easy to see that $\theta_{V_i, V_j}(z) \in \clb(\clw_j,
\clw)$ for all $z \in \D$ and $i \neq j$.

\NI Note that since the defect operator $D_{V_j} = 0$, the
characteristic function $\theta_{V_j}$ of $V_j$, $j = 1, 2$, is the
zero function. From this point of view, it is expected that the pair
of analytic invariants $\{\theta_{V_i, V_j} : i \neq j\}$ will
provide more information about pairs of commuting isometries.

\NI Subsequent theory for  pairs of commuting contractions and a
more detailed connection between pairs of commuting pure isometries
$(V_1, V_2)$ and the analytic invariants $\{\theta_{V_i, V_j} : i
\neq j\}$ as defined in (\ref{eqn-theta12}) and (\ref{eqn-theta21})
will be exhibited in more details in future occasion.

\vspace{0.3in}

\subsection*{Acknowledgements}
The authors are grateful to the
anonymous reviewers for their critical and constructive reviews and
suggestions that have substantially improved the manuscript. The
first author's research work is supported by NBHM Post Doctoral
Fellowship No. 2/40(50)/2015/ R \& D - II/11569. The research of the
second author was supported in part by (1) National Board of Higher
Mathematics (NBHM), India, grant NBHM/R.P.64/2014, and (2)
Mathematical Research Impact Centric Support (MATRICS) grant, File
No : MTR/2017/000522, by the Science and Engineering Research Board
(SERB), Department of Science \& Technology (DST), Government of
India.

\end{document}